\documentclass[aps,preprint,a4paper,floatfix]{revtex4}

\usepackage{graphicx,color}
\usepackage{amsmath,amssymb,amsfonts,amsthm,amscd,bm}
\usepackage{booktabs}
\usepackage{cancel}

\usepackage{amsmath}
\usepackage{amssymb}
\usepackage{amsfonts}
\usepackage{amsthm}
\usepackage{amscd}
\usepackage{bm}
\usepackage{graphicx}
\usepackage{booktabs}
\usepackage{listings}
\usepackage{xcolor,colortbl}
\usepackage{float}

\def\hat{\widehat}
\def\*{\star}
\def\[{\left[}
\def\]{\right]}
\def\({\left(}      

\def\){\right)}

\def\frac#1#2{\dfrac{#1}{#2}}
\def\inv#1{\dfrac{1}{#1}}

\def\2pi{\hbox{$2\pi i$}}

\def\dsl{\raise.15ex\hbox{/}\kern-.57em\partial}
\def\Dsl{\,\raise.15ex\hbox{/}\mkern-.13.5mu D}
\def\2pi{\hbox{$2\pi i$}}

\def\dsl{\raise.15ex\hbox{/}\kern-.57em\partial}
\def\Dsl{\,\raise.15ex\hbox{/}\mkern-.13.5mu D}
\font\numbers=cmss12
\font\upright=cmu10 scaled\magstep1
\def\stroke{\vrule height8pt width0.4pt depth-0.1pt}
\def\topfleck{\vrule height8pt width0.5pt depth-5.9pt}
\def\botfleck{\vrule height2pt width0.5pt depth0.1pt}
\def\Zmath{\vcenter{\hbox{\numbers\rlap{\rlap{Z}\kern
    0.8pt\topfleck}\kern 2.2pt
    \rlap Z\kern 6pt\botfleck\kern 1pt}}}
\def\Qmath{
    \vcenter{\hbox{\upright\rlap{\rlap{Q}\kern3.8pt\stroke}\phantom{Q}}}}
\def\Nmath{\vcenter{\hbox{\upright\rlap{I}\kern 1.7pt N}}}
\def\Cmath{\vcenter{\hbox{\upright\rlap{\rlap{C}\kern
                   3.8pt\stroke}\phantom{C}}}}
\def\Rmath{\vcenter{\hbox{\upright\rlap{I}\kern 1.7pt R}}}
\def\Z{\ifmmode\Zmath\else$\Zmath$\fi}
\def\Q{\ifmmode\Qmath\else$\Qmath$\fi}
\def\N{\ifmmode\Nmath\else$\Nmath$\fi}
\def\C{\ifmmode\Cmath\else$\Cmath$\fi}
\def\R{\ifmmode\Rmath\else$\Rmath$\fi}
\def\barray{\begin{eqnarray}}
\def\earray{\end{eqnarray}}
\def\beq{\begin{equation}}
\def\eeq{\end{equation}}

\def\Li{{\rm Li}}

\def\AA{\leavevmode\setbox0=\hbox{h}
\dimen0=\ht0 \advance\dimen0 by-1ex\rlap{\raise.67\dimen0\hbox{\char'27}}A}

\def\fluc{\mathfrak{f}}
\def\fluch{\mathfrak{h}}
\def\fluchhat{\hat{\fluch}}

\def\Li{{\rm Li}}
\def\fluchat{\hat{\fluc}}
\def\cCramer{c}
\def\kappa{c}

\makeatletter
\def\iddots{\mathinner{\mkern1mu\raise\p@
\vbox{\kern7\p@\hbox{.}}\mkern2mu
\raise4\p@\hbox{.}\mkern2mu\raise7\p@\hbox{.}\mkern1mu}}
\makeatother

\def\Li{{\rm Li}}

\theoremstyle{plain}
\newtheorem{theorem}{Theorem}
\newtheorem{lemma}{Lemma}
\newtheorem{proposition}{Proposition}

\theoremstyle{remark}

\begin{document}

\title{
Asymptotic upper bound  on prime gaps
}
\author{
 Andr\'e  LeClair\footnote{andre.leclair@gmail.com}
}
\affiliation{Cornell University, Physics Department, Ithaca, NY 14850} 

\begin{abstract}
The Cram\'er-Granville conjecture is an upper bound on prime gaps,
$g_n = p_{n+1} - p_n < \cCramer  \,  \log^2 p_n$ for some constant $\cCramer \geq 1$.   
Using a formula of Selberg,  we first prove the weaker summed version:
$\sum_{n=1}^N g_n <  \sum_{n=1}^N \log^2 p_n$.   
In the remainder of the paper we investigate which properties of the
fluctuations $\fluc (x) = \pi (x) - \Li(x)$ would imply the Cram\'er-Granville conjecture
is true and present two such properties,  one of which assumes the Riemann Hypothesis;  
however we are unable to prove these properties are indeed satisfied.   
We argue that the conjecture is related to the enormity of the Skewes number. 
\end{abstract}

\maketitle

Let $p_n$ denote the $n$-th prime number with $p_1=2$,  and define
the gaps 
\beq
\label{gap}
g_n = p_{n+1} - p_n 
\eeq
The lowest gap is obviously equal to 2,  however an interesting question
is how often this minimal gap occurs,   and the twin primes conjecture
says  an infinite number of times.    There has recently been
progress on this problem by Yitang Zhang \cite{Zhang}.    In the other direction,
upper bounds on gaps are also of interest.   In practice the latter are 
more important since,  given knowledge of a prime,   they  can aid in the location of the next prime.      

The Prime Number Theorem leads to $p_n \approx  n \log n$,  which
implies the average gap $g_n$ is $\log n$.    However it is known that
the maximal gaps grow faster than this \cite{Westzynthius}: 
$$\limsup_{n\to\infty}\frac{p_{n+1}-p_n}{\log p_n}=\infty$$
The Cram\'er-Granville conjecture \cite{Cramer,Granville} is the statement
\beq
\label{CG}
g_n  <  \cCramer  \, \log^2 p_n   , ~~~~~\forall n 
\eeq
for some constant $\cCramer\geq 1$.   Cram\'er's model supports the 
conjecture with $\cCramer=1$,   whereas Granville proposed that 
$\cCramer >1$ and suggested  $\cCramer \ge 2e^{-\gamma}\approx1.1229\ldots$.   
There is extensive numerical evidence for the conjecture  
\cite{Cadwell}.  
Furthermore,  thus far  this evidence supports the value $\cCramer=1$  at least for $n>4$:    
 the greatest known value of the ratio $g_n/\log^2 p_n$  is 0.9206386... for the prime 1693182318746371,  which is somewhere around the
  $n=5 \times 10^{13}$-th prime. 
 
 As usual let $\pi (x)$  denote the number of primes less than or equal 
 to $x$.   Riemann derived an explicit formula for $\pi (x)$ in terms of
 an infinite sum over zeros of $\zeta (s)$ inside the critical strip
 $0 \geq \Re (s) \leq 1$.   The Prime Number Theorem (PNT) is the 
 result that the leading term is 
 \beq
 \label{PNT} 
 \pi (x) =  \sum_{p\leq x}  1  \approx \int_2^x  \frac{dt}{\log t}   = \Li (x)
 \eeq
 It was proven independently by Hadamard and  de la  Vall\'ee Poussin 
 using Riemann's formula for $\pi (x)$ and showing that there are no
 zeros with $\Re (s) =1$.

Cram\'er's original conjecture is
essentially based on the PNT.   The Cram\'er model is 
 a probabilistic model of the primes, in which one assumes that the probability of a natural number of size x being prime is 1/log x.  Cram\'er proved that in this model, the  conjecture holds true with probability one.
Since there is a great deal of  numerical evidence the conjecture is correct,  
this suggests that a proof of it may only require something slightly 
stronger than the PNT,  and this will what we will investigate here. 
     A result of this kind is due to Selberg \cite{Selberg}.
The latter led to an independent,  so-called  elementary proof of the PNT,  in that it does
not rely on Riemann's formula for $\pi (x)$.   
The starting point for this article  is a formula of  Selberg,  namely
Theorem 1 below.

The Cram\'er-Granville conjecture may be difficult, or even impossible,
to prove for a single  given $n$.    On the other hand,  if it is known to be
true by direct computation for all $n < N_0$ for some high enough 
$N_0$,  one can attempt to prove the conjecture with a bootstrap 
principle that extrapolates to infinity using some asymptotic formulas
valid in the limit of large $N$.   The asymptotic formulas must be such
that the relative fluctuations decrease fast enough not to spoil the
validity of the conjecture.  
This is in analogy to statistical physics of a system of
$N$ particles,   where the relative fluctuations are typically of order
$1/\sqrt{N}$ so that results just get better and better as one increases $N$,
i.e. in the so-called  thermodynamic limit.   

In the sequel,  relative fluctuations  will  be quantified as follows.
Since $\Li (x)$ is the leading term,  
\beq
\label{fluctu}
 \pi (x) =  \Li (x)   +  \fluc (x) 
 \eeq
 where the fluctuating term $\fluc (x)$ grows more slowly 
 than  the smooth part $\Li (x) \approx  x/\log x$,  i.e. $\lim_{x \to \infty}  \fluc(x)/\Li(x) = 0$. 
We will consider different bounds on $\fluc (x)$ depending on whether one assumes the
Riemann Hypothesis or not.   

\bigskip
\bigskip

We begin with:

\bigskip
\begin{theorem}\label{SelbergTh}
  (Selberg) 
\beq
\label{Selberg}
2 x \log x + O(x)  =  \sum_{p\leq x}  \log^2 p   ~~+   \sum_{p,q {~ \rm with~}  pq \leq x} \log p \log q 
\eeq
where $p,q$ are primes. 
\end{theorem}  

\bigskip
Let us first make some simple observations based on the above theorem.   
Let $S_1 (x)$  denote the first sum on the RHS of the above equation 
and $S_2 (x)$ the second sum.  
 $S_1$ includes terms up to 
$\log^2 x$,   whereas  $S_2$ contains terms with $p=q$  up to  $\log^2  \sqrt{x}$,  which are much smaller.  
There are additional terms in $S_2$ compared to $S_1$ for $p >  q$, 
where the largest prime is approximately $x/2$ corresponding to the 
term $\log (x/2) \log 2 $,  and they are also considerably smaller.       This strongly suggests that $S_1 (x) >  S_2 (x)$.  
One can easily check numerically that for all $x< p_{10^4} = 104729$,  
$S_1 (x)$  is significantly  larger than $S_2 (x)$ and the difference increases with $x$. 
    For instance, for  $x=p_{10^4}$,  
$S_1 (x)  - S_2 (x) =  686787.25..$.     For larger $x$,  the difference 
$S_1 - S_2$  only continues to grow.       These considerations lead us to 
formulate the following lemma:

\bigskip
\noindent
\begin{lemma}\label{Lemma1}
\beq
\label{lemma}
   \sum_{p,q {~ \rm with~}  pq \leq x} \log p \log q
~< ~ \sum_{p\leq x}  \log^2 p 
\eeq
\end{lemma}
\begin{proof}
 For  $x< p_{N_0}$ where $N_0 = 10^4$ for instance or even much smaller, 
 the above
inequality is easily verified by direct computation.   For higher $x$,  
the relative fluctuations of $S_1$ and $S_2$ are very small compared to
the difference $S_1 - S_2$.     Thus, to go to higher $x$,  one can use
the following asymptotic formulas:
\barray
\label{smoothS}
\sum_{p\leq x}  \log^2 p &\approx&  \int_2^x  \frac{dt}{\log t} ~\[ \log^2 t \] =  
x\log x - x - 2 \log 2 + 2
\\ 
\nonumber
 \sum_{p,q {~ \rm with~}  pq \leq x}   \log p \log q 
 &\approx&
 \int_2^x  \frac{dt}{\log t}  \int_2^{x/t}  \frac{du}{\log u}  ~ \[  \log t  \log u \] = 
 x\log x - (2 + \log 2) x + 4
 \earray
 Therefore $S_1 (x) -  S_2 (x) \approx (1+\log 2) x$ for large $x$,
 and this linear growth in $x$ is much larger  than the fluctuations
 for large enough $x$.  
\end{proof}

\bigskip
\noindent
\begin{theorem}\label{theorem2}
  For large enough $N>N_0$ for some finite $N_0$,  one has
\beq
\label{gaptheorem}
\sum_{n=1}^N   g_n   ~<~  \sum_{n=1}^N  \log^2 p_n  \eeq
\end{theorem}
\begin{proof}
 From  Lemma 1 and Theorem 1,  one has 
\beq
\label{gap1}
x \log x +O(x)  ~<~   \sum_{p \leq x}  \log^2 p   
\eeq
Let $x = p_{N+1} - \epsilon$  with $\epsilon$ small and positive,
and consider the limit $\epsilon \to 0$.      
Then \eqref{gap1} can be expressed as 
\beq
\label{SelP}
p_{N+1} \log p_{N+1} +O(p_{N+1}) ~<~  \sum_{n=1}^N \log^2 p_n 
\eeq  
It should be kept in mind that the $O(p_{N+1})$  term can be negative here.   Nevertheless,  
 for large enough $N$, 
 $p_{N+1} <  p_{N+1} \log p_{N+1} +O(p_{N+1})$. 
Next,  noting that   
$p_{N+1} = \sum_{n=1}^N  g_n    +  2$
 proves the theorem.    
 \end{proof}

\bigskip

Theorem 2 is consistent with the Cram\'er-Granville conjecture,
and suggests $\cCramer =1$,   although it is also consistent with $\cCramer >1$ 
since the fluctuations in the sum   could  just average out  to give  $\cCramer =1$.        
Furthermore  it  is clearly  not enough to establish it since it does not imply 
that each individual term in the sum satisfies the inequality.
For the remainder of this article,  we propose conditions  on the fluctuations 
which would imply the Cram\'er-Granville conjecture,   however we are unable to 
prove these conditions are true.

\bigskip

\noindent

\begin{proposition}\label{Deltamon}
Define 
\beq
\label{Deltax}
\Delta (x)  =  \sum_{p<x}  \(  \log^2( p) - \frac{g(p)}{\kappa} \)
\eeq
where $g(p_n) =  g_n$.   
  Then  if $\Delta (x)$  with $\kappa >1$ is a monotonically increasing function 
of $x$  in a region $x_1 < x < x_2$ 
  then the Cram\'er-Granville conjecture is true   for all primes  $p$ 
  in the region  
$x_1 < p < x_2$.   
\end{proposition}

\begin{proof}
Define the  discrete function of $N$ 
\beq
\label{DN}
D(N) \equiv  \sum_{n=1}^N  \( \log^2 p_n - \frac{g_n}{\kappa} \)  
\eeq
Now $D(N)$ may be viewed as the above  function $\Delta$ of $p_{N+1}$,  i.e. 
$  D(N) = \Delta (p_{N+1})$.     
Since $p_{N+1}$ is a monotonically increasing function of $N$,   then 
if $\Delta (p)$ is a monotonically increasing function of $p$ for large enough $p$,  then $D(N)$ 
is a monotonically increasing function of $N$ for large enough $N>N_0$.  
This  in turn  implies  that for $N>N_0$,   at each step in the sum,  $N-1  \to N$,  one 
has  $\kappa \log^2 p_N > g_N$.
\end{proof}

\begin{proposition}\label{hofx}
\beq
\label{Deltap}
\Delta (x)= x \log x -  \frac{(c+1)}{c}  x + \fluchhat (x)  +O(1)
\eeq
where 
\beq
\label{hx}
- \frac{B x}{\log x}  <  \fluchhat (x) <   \frac{B x}{\log x} 
\eeq
for some constant $B\approx 5$.   
\end{proposition} 

\bigskip
\begin{proof}
    One has the exact formula
\beq
\label{logsqu}
\sum_{p<x}  \log^2 p  =  \int_2^x  d\pi (t)  \log^2 t 
\eeq
It is known that  \cite{Dusart,Dusart2} 
\beq
\label{piBound}
 \frac{x}{\log x}   +  \frac{x}{\log^2 x}   + \frac{1.8 x}{\log^3 x }< \pi (x)  <
 \frac{x}{\log x}   +  \frac{x}{\log^2 x}   +   \frac{2.51 x}{\log^3 x } 
\eeq
 for large enough $x$;   the first inequality requires $x\geq 32299$ and the 
 second $x \geq 355991$.        
Thus we can simply write this as 
\beq
\label{PiO}
\pi (x) = \frac{x}{\log x}   +  \frac{x}{\log^2 x}   +  \frac{2x}{\log^3 x} +  O\(    \frac{x}{\log^3 x }  \)
\eeq
where the first three terms come from the expansion of $\Li (x)$.   
The bound 
 \eqref{PiO}  is only valid for large enough $x$.   However by changing the constants 
 one can obtain a bound valid for all $x$:
\beq
\label{pibound}
\pi (x) = \frac{x}{\log x} + \frac{x}{\log^2 x}  + \frac{2x}{\log^3 x } + \fluchat (x) 
\eeq 
where 
\beq
\label{fluchatbound}
- \frac{Bx}{\log^3 x}   <  \fluchat (x) <   \frac{Bx}{\log^3 x} 
\eeq
for some constant $B \approx 5$.  We determined this value of $B$  by verifying 
 \eqref{fluchatbound} is valid for $x$ below the value where \eqref{PiO} becomes valid.   Integrating  \eqref{logsqu} by parts,   this implies 
\eqref{hx}.
\end{proof}

Define  the function $b(x)$ such that  
\beq
\label{Cx}
\fluchhat (x) =  \frac{b(x) x}{\log x}  
\eeq
where $-B < b(x) < B$.  
One has 
\beq
\label{DelDer}
\frac{d\Delta}{dx}  = \log x -\inv{\kappa} + \frac{b(x)}{\log x} \( 1- \inv{\log x} \) 
+ \frac{b' (x) x}{\log x} 
\eeq
where $b'(x)$ is formally $db(x)/dx$.    Due to the jump discontinuities 
of $b(x)$ where the derivative is not defined,  one needs to formally define $b'(x)$.  
On a prime $p_n$ let us define it as follows:
\beq
\label{Cprime}
b' (p_n) = \frac{b(p_{n+1}) - b(p_n)}{p_{n+1} - p_n }  
\eeq
We then define $b'(x)$ for other $x$ as a linear  interpolation between $b'(p_n)$, $n=1,2,3....$.   

Recall that based on Proposition  \ref{Deltamon},  we wish to 
show that $d\Delta (x)/dx > 0$ for $x> x_0$ for some finite $x_0$.  
Let us first provide some  evidence that the $b'$ term in 
\eqref{DelDer}  can be neglected,  i.e. does not spoil the monotonicity of $\Delta (x)$ at high $x$.  
  To this end, let us first  assume $b' = 0$ so that 
$b$ is  constant.  
  Then as far as monotonicity goes,  
   in  the worse case  $b(x) = -B$,
and $\Delta (x)$ would be monotonically increasing for 
\beq
\label{xlower} 
x \gtrsim e^{1/2\kappa+ \sqrt{1/4\kappa^2 +B}}
\eeq
Note that for a given $B$,   the above equation does not 
single out any particularly special value of $\kappa$.   However,  if, for instance,  one wishes 
the Cram\'er-Granville conjecture to be true beyond the first $5$ primes,
then  $\kappa \approx 1$,   since  
  for $B=5$,  \eqref{xlower}  gives $x> 16.3 > p_6$.     
  Interestingly,   the original  Cram\'er conjecture with $\cCramer=1$ is
only known to fail for the first $4$ primes,   and beyond this has been verified numerically to
very high $N>10^{10}$ or so.    If we chose a smaller value of $\kappa$,   
then $\Delta (x)$ simply  becomes monotonically increasing beyond a higher value of $x$.   
The  fact that this analysis led to the very  likely  correct prediction 
that the Cram\'er-Granville conjecture is valid for primes $p>p_5$,    suggests that the 
$b' (x)$ must indeed be very small.      As we discuss below,  this appears
to be related to the enormous value of the Skewes number.     
This analysis also suggests the value of $\cCramer$ is fixed by the 
gaps in the {\it low} primes.

One can state something more precise as follows.    Only if $b'(x) <  0$ can it spoil the
monotonicity.    
If the following condition holds
\beq
\label{cprimebound}
b'(p) > - \frac{\log^2 p}{p} \( 1- \inv{\kappa \log p}  \)
\eeq
where $p$ is  prime with $p>x_0$ 
 for some $x_0$,  then the Cram\'er-Granville conjecture is true 
for $p> x_0$.     
 Numerically,  we checked that the above condition is valid  for 
   $5< p < 10^{12}$  with $\kappa =1$,  however we cannot prove that it is valid beyond this.  
 
Stronger bounds on the fluctuations do not significantly improve the analysis.     
Since it is believed that the strongest bounds come from assuming the 
Riemann Hypothesis,  let us do so.    Schoenfeld showed \cite{Schoenfeld} 
that 
\beq
\label{Scho}
| \fluc (x) | = | \pi (x) - \Li (x) |  <   K \sqrt{x} \log x
\eeq
for $x>2657$ where $K=1/8\pi$.      By increasing $K$,  one can make the above 
bound valid for all $x$;    we found $K\approx 1/3$.   
Repeating the above steps leads to 
\beq
\label{DeltaRH}
\Delta (x) = x \log x - \frac{(c+1)}{c}  x + \fluch (x) +O(1)
\eeq
where 
\beq
\label{fluch}
\fluch (x) =  k(x) \sqrt{x}  \log^3 x
\eeq
From \eqref{Scho} one has 
\beq
\label{Kbound} 
- K  <  k(x) <  K 
\eeq
In \eqref{fluch} we have dropped a few terms with lower powers of $\log x$, 
such as $\sqrt{x} \log^2 x$  since they will not   significantly change our conclusions.

Define the formal derivative $k'(x)$ of $k(x)$ as in \eqref{Cprime} with $b\to k$.  
In Figure \ref{kprime} we plot some values.      
We then  have 
\beq
\label{DelDerRH}
\frac{d \Delta (x)}{dx} =  \log x -\inv{\kappa} + \frac{k(x) \log x}{\sqrt{x}} +  k'(x) \sqrt{x} \log^3 x
\eeq
As before, 
if one first assumes the $k' $ term is negligibly small,   then the worse case is 
$k = -K$.    One finds that for $\kappa =1$ and $K=1/3$,    $d\Delta (x) / dx > 0 $ for $x>p_2$. 
Thus,  once again we find that neglecting the $k'$ term apparently  leads to the right conclusion,  i.e. 
that the Cram\'er-Granville conjecture is true with $\kappa =1$ for all primes $p> p_4$.   
In particular,  if the following condition holds
\beq
\label{kprimebound}
k'(p) > - \inv{\sqrt{p} \log^2 p} \( 1- \inv{\kappa \log p} \)  
\eeq
for $p>x_0$ then the Cram\'er-Granville conjecture holds for  primes $p>x_0$.   
We checked numerically that the above condition  with $\kappa =1$   is satisfied for all primes
$p>3$ up to $10^{12}$,  however again we cannot prove this result.     See Figure \ref{kprime}. 

\begin{figure}[t]
\centering\includegraphics[width=.5\textwidth]{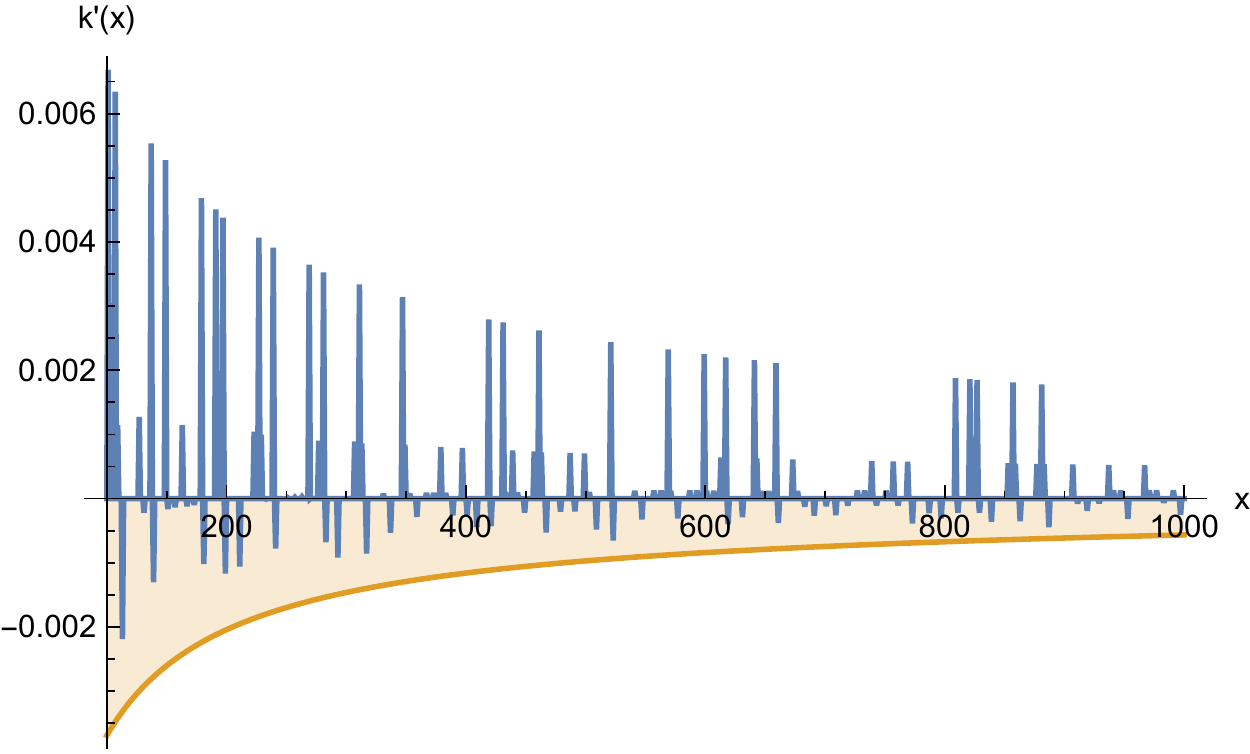}
\caption{$k'(x)$ evaluated at primes.   The smooth curve is the RHS of \eqref{kprimebound}
with $c=1$. }
\label{kprime}
\end{figure}

\def\Sk{{\rm Sk}}

\bigskip

Additional evidence for $k'(x)$ being very small  comes from the largeness of the Skewes number.  
Up to 
values of $x$ that are within reach numerically,  $\Li (x) > \pi (x)$,  and for a long time
it was believed that this persists to infinity.   
 Let $\Sk_1$ denote the Skewes number,  which is the
first $x$ where the  crossover    $\Li (x) < \pi (x)$ occurs.
   Littlewood proved $\Sk_1$ exists, i.e. isn't infinite,  and Skewes first  estimated it 
as 
$\Sk_1 <10^{10^{10^{34}}}$.
 This has since been reduced to $\Sk_1 < 10^{316}$ \cite{Bays}.   
 Numerically it is also known that $\Sk_1 > 10^{14}$.   
Now   $k (x)$ starts out negative with $k(x)>-1/8\pi$ for low $x$,  
then very slowly increases on average until it finally changes sign at $\Sk_1$. 
Thus the average of $k'(x)$  in the region $2< x < \Sk_1$ is approximately 
$(8\pi \Sk_1)^{-1}$,   which is exceedingly small.

   It is interesting to try and determine the Skewes number from the smallnes of 
   $k'(x)$.   
Numerically we find that $k(x) = O(\log \log \log x)$,   so that $k'(x)$ is indeed 
small at large $x$.         In particular,   For $x<10^{10}$ 
a reasonably good fit is that on average 
\beq
\label{kfit}
k(x)  \approx - A( \alpha - \log \log \log x )  
\eeq
with $\alpha \approx 1.3$.   
If  this  approximation persists,  then $\Sk_1 = e^{e^{e^\alpha}}$.    This is obviously very sensitive
to $\alpha$.    For $\alpha = 1.3$,  $\Sk_1 =  10^{17}$,  however a small increase to 
$\alpha = 1.5$ would already give $\Sk_1 = 10^{38}$.     Any values $\alpha >2$ are 
already ruled out since for $\alpha =2$,  $\Sk_1 > 10^{702}$.

\section*{Acknowledgments}

We wish to thank Guilherme Fran\c ca and Hendrik Vogt  for discussions.

\end{document}